\documentclass[11pt]{amsart}
\usepackage{amsmath,amsthm,amssymb, amscd, amsfonts}
\usepackage[all]{xy}
\usepackage{leftidx}
\usepackage[latin1]{inputenc}
\usepackage{enumerate}
\usepackage[dvipsnames]{xcolor}

\theoremstyle{plain}

\newtheorem{theorem}{Theorem}[section]
\newtheorem{corollary}[theorem]{Corollary}

\newtheorem{proposition}[theorem]{Proposition}
\newtheorem{lemma}[theorem]{Lemma}

\theoremstyle{definition}

\theoremstyle{remark}
\newtheorem{remark}[theorem]{Remark}

\numberwithin{equation}{section}\theoremstyle{plain}

\newcommand{\I}{\mathcal{I}}

\renewcommand{\1}{\textbf{1}}

\newcommand{\A}{{\mathcal A}}
\newcommand{\B}{{\mathcal B}}
\newcommand{\C}{{\mathcal C}}
\newcommand{\D}{{\mathcal D}}

\newcommand{\cd}{\mathrm{cd}}

\newcommand\Irr{\operatorname{Irr}}
\newcommand\FPdim{\operatorname{FPdim}}

\newcommand\vect{\operatorname{Vec}}

\newcommand\rk{\operatorname{rk}}

\begin{document}
\title{Slightly trivial extensions of a fusion category}
\author{Jingcheng Dong}
\email{jcdong@nuist.edu.cn}
\address{College of Mathematics and Statistics, Nanjing University of Information Science and Technology, Nanjing 210044, China}

\keywords{Fusion category; Extension; Ising category}

\subjclass[2010]{18D10}

\date{\today}

\begin{abstract}
We introduce and study the notion of slightly trivial extensions of a fusion category which can be viewed as the first level of complexity of extensions. We also provide two examples of slightly trivial extensions which arise from rank $3$ fusion categories.
\end{abstract}

\maketitle

\section{Introduction}\label{sec1}
Let $G$ be a finite group and $e$ be the identity of $G$. A fusion category $\C$ is $G$-graded if there is a decomposition $\C=\oplus_{g\in G}\C_g$ of $\C$ into a direct sum of full abelian subcategories, such that $(\C_g)^{*}=\C_{g^{-1}}$ and $\C_g\otimes \C_h\subseteq \C_{gh}$ for all $g,h \in G$.

The grading $\C=\oplus_{g\in G}\C_g$ is called faithful if $\C_g\neq 0$ for all $g\in G$. A fusion category $\C$ is called a $G$-extension of $\D$ if there is a faithful grading $\C=\oplus_{g\in G}\C_g$, such that the trivial component $\C_e$ is equivalent to $\D$. If $G$ is trivial then $\C$ is equivalent to $\D$. In this case, we call $\C$ is a trivial extension of $\D$.

In this paper, we introduce the new notion of slightly trivial extensions of a fusion category, which can be viewed as the first level of complexity of extensions, except the trivial extension. Let $\C=\oplus_{g\in G}\C_g$ be a $G$-extension of a fusion category $\D$. We obtain a necessary and sufficient condition for $\C$ being a slightly trivial extension of $\D$. In particular, if the largest pointed fusion subcategory $\D_{pt}$ is trivial and $\C$ is a slightly trivial extension of $\D$, then $\C=\C_{pt}\bullet \D$ admits an exact factorization. When $\C$ is a slightly trivial extension of $\D$ and the Grothendieck ring of $\C$ is commutative, we also obtain the fusion rules of $\C$  in terms of that of $\D$. In the last section of this paper, we provide two examples of slightly trivial extensions which arise from rank $3$ fusion categories.

\section{Preliminaries}\label{sec2}
\subsection{Frobenius-Perron dimensions}
Let $\C$ be a fusion category. We shall denote by $\Irr(\C)$ the set of isomorphism classes of simple objects in $\C$. The set $\Irr(\C)$ is a basis of the Grothendieck ring of $\C$. The cardinal number of $\Irr(\C)$ is called the rank of $\C$ and is denoted by $\rk(\C)$.

Let $\C$ be a fusion category and let $X$ be a simple object of $\C$. The Frobenius-Perron dimension $\FPdim(X)$ of $X$ is defined as the Frobenius-Perron eigenvalue of
the matrix of left multiplication by the class of $X$ in the basis $\Irr(\C)$. The Frobenius-Perron dimension of $\C$ is $\FPdim (\C)=\oplus_{X \in \Irr(\C)} (\FPdim X)^2$. The following lemma is due to M\"{u}ger, see \cite[Remark 8.4]{etingof2005fusion}.
\begin{lemma}\label{dim}
Let $X$ be a simple object of a fusion category. Then $\FPdim(X)\geq 1$. In particular, if $\FPdim(X)<2$ then $\FPdim(X)=2\cos(\frac{\pi}{n})$ for some integer $n\geq 3$.
\end{lemma}

A simple object $X$ of a fusion category is invertible if $X\otimes X^*=\1$, where $X^*$ is the dual of $X$ and $\1$ is the trivial simple object of $\C$. It is obvious that a simple object $X$ is invertible if and only if $\FPdim(X)=1$. A fusion category is pointed if all of its simple objects are invertible. Let $\C$ be a fusion category. We shall denote by $\C_{pt}$ the largest pointed fusion subcategory of $\C$, and by $G(\C)$ the group of isomorphism classes of invertible objects of $\C$. It is clear that $G(\C)$ generates $\C_{pt}$ as a fusion subcategory.

The class of pointed fusion categories has been classified (see e. g. \cite{Ostrik2003}). Let $G$ be a finite group, $\omega$ be a cohomology class in $H^3(G,k^*)$, and $\vect_G^{\omega}$ be the category of finite dimensional vector spaces graded by $G$ with associativity determined by $\omega$. It is shown that a pointed fusion category is equivalent to some $\vect_G^{\omega}$.



\subsection{Exact factorizations of fusion categories}
Let $\C$ be a fusion category, and let $\A, \B$ be fusion subcategories of $\C$. Let $\A\B$ be the full abelian (not necessarily tensor) subcategory of $\C$ spanned by direct summands in $X\otimes Y$, where $X\in \A$ and $Y\in \B$. We say that $\C$ factorizes into a product of $\A$ and $\B$ if $\C=\A\B$. A factorization $\C=\A\B$ of $\C$ is called exact if $A\cap \B=\vect$, and is denoted by $\C=\A\bullet\B$, see \cite{gelaki2017exact}.

By \cite[Theorem 3.8]{gelaki2017exact}, $\C=\A\bullet\B$ is an exact factorization if and only every simple object of $\C$ can be uniquely expressed in the form $X\otimes Y$, where $X\in \Irr(\A)$ and $\Irr(\B)$.

\section{Slightly trivial extensions}\label{sec3}

Let $\D$ be a fusion category with $\Irr(\D)=\{\1=X_0,X_1,\cdots,X_n\}$. Let $\C=\oplus_{g\in G}\C_g$ be an extension of $\D$. We say that the component $\C_g$ is \emph{similar} to $\D$ if there exists an invertible simple object $\delta_g\in \C_g$ such that $\Irr(\C_g)=\{\delta_g=\delta_g\otimes \1,\delta_g\otimes X_1,\cdots,\delta_g\otimes X_n\}$. If all components are similar to $\D$ then we say that $\C$ is a \emph{slightly trivial} extension of $\D$.

\begin{proposition}\label{invertibles}
Let $\C=\oplus_{g\in G}\C_g$ be an extension of a fusion category $\D$. Then the component $\C_g$ is similar to $\D$ if and only if $\C_g$ contains an invertible simple object. In particular, $\C$ is a slightly trivial extension of $\D$ if and only if every component contains an invertible simple object.
\end{proposition}
\begin{proof}
Assume that $\C_g$ contains an invertible simple object $\delta_g$. Let $\Irr(\D)=\{\1,X_1,\cdots,X_n\}$. Then $\delta_g,\delta_g\otimes X_1,\cdots,\delta_g\otimes X_n$ are nonisomorphic simple objects in $\C_g$. Since $\FPdim(\delta_g)=\FPdim(\1)$,  $\FPdim(\delta_g\otimes X_i)=\FPdim(X_i)$ and $\FPdim(\D)=\FPdim(\C_g)$, we conclude that $\Irr(\C_g)=\{\delta_g,\delta_g\otimes X_1,\cdots,\delta_g\otimes X_n\}$. Hence $\C_g$ is similar to $\D$.

The other direction is obvious, by the definition of a slightly trivial extension.
\end{proof}

Let $1 = d_0 < d_1< \cdots < d_s$ be positive real numbers, and let $n_0,n_1,\cdots,n_s$ be positive integers. A fusion category is said of type $(d_0,n_0; d_1,n_1;\cdots;d_s,n_s)$ if $n_i$ is the number of the non-isomorphic simple objects of Frobenius-Perron dimension $d_i$, for all $0\leq i\leq s$.

\begin{remark}\label{rmk1}
Let $\C=\oplus_{g\in G}\C_g$ be a slightly trivial extension of a fusion category $\D$.

(1)\ Every component $\C_g$ contains an invertible simple object $\delta_g$. Set $\Irr(\D)=\{\1,X_1,\cdots,X_n\}$ and $\Delta_G=\{\delta_g\in \C_g|g\in G\}$, where $\delta_e=\1$. The proof of Proposition \ref{invertibles} shows that $\Irr(\C)=\{\delta_g,\delta_g\otimes X_1,\cdots,\delta_g\otimes X_n|\delta_g\in \Delta_G\}$.

(2)\ The proof of Proposition \ref{invertibles} also shows that $\C_g$ and $\C_h$ have the same type, for all $g,h\in G$.
\end{remark}

\begin{corollary}\label{exact1}
Let $\C=\oplus_{g\in G}\C_g$ be a slightly trivial extension of a fusion category $\D$. Assume that $\D_{pt}$ is trivial. Then $\C=\C_{pt}\bullet \D$ is an exact factorization of $\C_{pt}$ and $\D$.
\end{corollary}
\begin{proof}
Since $\D_{pt}$ is trivial, the proof of Theorem \ref{invertibles} shows that every component $\C_g$ exactly contains only one invertible simple object. Let $\delta_g$ be the invertible simple object in $\C_g$. Then $\Delta_G=\{\delta_g|g\in G\}=G(\C)$. Hence every simple object of $\C$ can be expressed in the form $X\otimes Y$, where $X\in \C_{pt}$ and $Y\in \D$ are simple objects, by Remark \ref{rmk1}. The result then follows from \cite[Theorem 3.8]{gelaki2017exact}.
\end{proof}

%
%

Let $\C=\oplus_{g\in G}\C_g$ be an extension of a fusion category $\D$. If $\delta_g\in \C_g,\delta_h\in \C_h$ are invertible simple objects then $\delta_g\otimes \delta_h\in \C_{gh}$ is also an invertible simple object. We shall write $\delta_{gh}=\delta_g\otimes \delta_h$ in the following corollary.
\begin{corollary}\label{exact2}
Let $\C=\oplus_{g\in G}\C_g$ be a slightly trivial extension of a fusion category $\D$. Suppose that the Grothendieck ring of $\C$ is commutative. Then the fusion rules of $\C$ are determined by that of $\D$.
\end{corollary}
\begin{proof}
Let $\Irr(\D)=\{\1=X_0,X_1,\cdots,X_n\}$. We assume the fusion rules of $\D$ are determined by the following equation:
\begin{equation}
\begin{split}
X_i\otimes X_j=\bigoplus_{t=0}^{n}N_{ij}^tX_t.
\end{split}\nonumber
\end{equation}

Let $\delta_g\in \C_g,\delta_h\in \C_h$ are invertible simple objects. Then $\delta_g,\delta_g\otimes X_1,\cdots,\delta_g\otimes X_n$ and $\delta_h,\delta_h\otimes X_1,\cdots,\delta_h\otimes X_n$ are all nonisomorphic simple objects in $\C_g$ and $\C_h$, respectively. Then
\begin{equation}
\begin{split}
(\delta_g\otimes X_i)\otimes (\delta_h\otimes X_j)&=(\delta_g\otimes \delta_h)\otimes (X_i\otimes X_j)\\
                  &=\delta_{gh}\otimes (X_i\otimes X_j)\\
                  &=\delta_{gh}\otimes\bigoplus_{t=0}^{n}N_{ij}^tX_t\\
                  &=\bigoplus_{t=0}^{n}N_{ij}^t\delta_{gh}\otimes X_t.
\end{split}\nonumber
\end{equation}
\end{proof}

\section{Examples of slightly trivial extensions}
In this section, we shall need the following lemma.
\begin{lemma}\cite[Proposition 8.20]{etingof2005fusion}\label{FPdim}
Let $\C=\oplus_{g\in G}\C_g$ be a $G$-extension of a fusion category $\D$. Then $\FPdim(\C_g)=\FPdim(\C_h)$ for all $g,h\in G$ and $\FPdim(\C)=|G|\FPdim(\D)$ .
\end{lemma}

For a fusion category $\C$, we shall denote by $\cd(\C)$ the set of Frobenius-Perron dimensions of simple objects of $\C$.

\begin{lemma}\label{low-simps}
Let $\D$ be a fusion category and  $\C=\oplus_{g\in G}\C_g$ be an extension of $\D$. Assume that $\D$ contains a simple object $X$ such that $1<\FPdim(X)<2$. Then, for every $g\in G$, component $\C_g$ has one of the following properties.

(1)\, $\C_g$ contains an invertible simple object.

(2)\, There exists $Y\in \Irr(\C_g)$ such that $X\otimes Y$ is a simple object, where $\FPdim(Y)$ is the smallest one in $cd(\C)$.
\end{lemma}
\begin{proof}
Assume that part (1) does not hold true and prove (2). In this case, $\FPdim(Y)\geq \sqrt{2}$ for every simple object $Y$ in $\C_g$, by Lemma \ref{dim}. Assume that $\Irr(\C_g)=\{X_1^g,\cdots,X_n^g\}$. We may reorder them such that $\FPdim(X_1^g)\leq\cdots\leq \FPdim(X_n^g)$. Since $X\otimes X_1^g\in \C_g$, we may set $X\otimes X_1^g=\oplus_{i=1}^nN^iX_i^g$. Since $\FPdim(X_1^g)<\FPdim(X\otimes X_1^g)<2\FPdim(X_1^g)$, we have $N^1=1$ or $0$. If $N^1=1$ then
$$ \FPdim(\oplus_{i=2}^nN^iX_i^g)=\FPdim(X\otimes X_1^g)-\FPdim(X_1^g)<\FPdim(X_1^g).$$

This is impossible since $\FPdim(X_i^g)\geq \FPdim(X_1^g)$ for all $i\geq 2$. Therefore, we have $N^1=0$ and $X\otimes X_1^g=\oplus_{i=2}^nN^iX_i^g$.

Since $\FPdim(X_i^g)\leq \FPdim(X\otimes X_1^g)<2\FPdim(X_i^g)$, for all $i\geq 2$, we have $N^i=1$ or $0$. Let $k$ be the smallest number such that $N^k=1$. Then
$$\FPdim(X\otimes X_1^g)-\FPdim(X_k^g)=\FPdim(\oplus_{i>k}^nN^iX_i^g)<\FPdim(X_k^g).$$

By the minimality of $k$, we get $N^i=0$ for all $i>k$.  Hence we have $X\otimes X_1^g=X_k^g$ is simple.
\end{proof}

Let $\D$ be a rank $3$ fusion category with $\Irr(\D)=\{\1,\alpha,\beta\}$. The fusion rules of $\D$ are determined by
\begin{equation}\label{A(1,5)}
\begin{split}
\alpha\otimes \alpha=\1 \oplus \beta,\quad \beta\otimes \beta=\1\oplus \alpha\oplus \beta,\quad \alpha\otimes \beta=\beta\otimes \alpha=\alpha\oplus \beta.
\end{split}
\end{equation}

It is not hard to check that $\FPdim(\alpha)=2\cos(\frac{\pi}{7})$  and $\FPdim(\beta)=4\cos(\frac{\pi}{7})^2-1$. It is proved  in \cite{raey} that $\D$ is a modular category.

\begin{theorem}\label{exten_A(1,5)}
Let $\D$ be a fusion category with the fusion rules (\ref{A(1,5)}) and let $\C=\oplus_{g\in G}\C_g$ be an extension of $\D$. Then

(1)\, $\C$ is a slightly trivial extension of $\D$.

(2)\, $\C=\C_{pt}\bullet\D$ is an exact factorization of $\C_{pt}$ and $\D$.
\end{theorem}
\begin{proof}
(1)\, By Theorem \ref{invertibles}, it suffices to prove that every component contains one invertible simple object.

Suppose on the contrary that there exists a component $\C_g$ such that every simple object in $\C_g$ is not invertible. Let $X\in \C_g$ be a simple object such that $\FPdim(X)$ is smallest in $cd(\C_g)$. Then Lemma \ref{low-simps} shows that $\alpha\otimes X$ is a simple object in $\C_g$. That is, $\C_g$ contains at least $2$ simple objects. By Lemma \ref{dim}, $\FPdim(X)\geq\sqrt{2}$ and hence $\FPdim(\alpha\otimes X)\geq2\sqrt{2}\cos(\frac{\pi}{7})$.

If $\rk(\C_g)\geq3$ then $\FPdim(\C_g)\geq 4+8\cos(\frac{\pi}{7})^2>\FPdim(\D)$. This is impossible by Lemma \ref{FPdim}. Hence $\C_g$ only contains $2$ simple objects.

Let $\Irr(\C_g)=\{X,Y\}$ with $\FPdim(X)\leq \FPdim(Y)$. By Lemma \ref{low-simps}, $Y=\alpha\otimes X$. By \cite[Lemma 4.10]{natale2012solvability}, the decompositions of $\alpha\otimes \alpha$ and $X\otimes X^*$ can not have common summands. This implies that $X\otimes X^*=\1\oplus \alpha$ and hence $\FPdim(X)=\sqrt{1+2\cos(\frac{\pi}{7})}$. So we have $\FPdim(\C_g)=\FPdim(X)^2+\FPdim(Y)^2=1+2\cos(\frac{\pi}{7})+4\cos(\frac{\pi}{7})^2+8\cos(\frac{\pi}{7})^3$, which is not equal to $\FPdim(\D)$. It is also impossible by Lemma \ref{FPdim}. This proves part (1).

(2)\, Part (2) follows from Part (1) and Corollary \ref{exact1}, since $\D_{pt}$ is trivial.
\end{proof}

\medbreak
An Ising category $\I$ is a fusion category which is not pointed and has Frobenius-Perron dimension $4$. Let $\I$ be an Ising category then $\Irr(\I)$ consists of three simple objects: $\1$, $\delta$ and $X$, where $\1$ is the trivial simple object, $\delta$ is an invertible object and $X$ is a non-invertible object. They obey the following fusion rules:
$$\delta\otimes \delta\cong \1,\delta\otimes X\cong X\otimes \delta \cong X,X\otimes X\cong \1\oplus\delta.$$

It is easy to check that $\FPdim(\delta)=1$, $\FPdim(X)=\sqrt{2}$ and $X$ is self-dual. It is known that any Ising category $\I$ is a modular category. See \cite[Appendix B]{drinfeld2010braided} for more details on Ising categories.

\begin{theorem}
Let $\C=\oplus_{g\in G}\C_g$ be an extension of an Ising category $\I$. Then $\C$ is a slightly trivial extension of $\I$.
\end{theorem}

\begin{proof}
As before, we set $\Irr(\I)=\{\1, \delta,X\}$. By Proposition \ref{invertibles}, it suffices to prove that every component $\C_g$ contains at least one invertible object.

Suppose on the contrary that every simple object in $\C_g$ is non-invertible. Let $X_g$ be a simple object in $\C_g$ such that $\FPdim(X_g)$ is smallest in $cd(\C_g)$. Then $X \otimes X_g$ is simple by Lemma \ref{low-simps}. By Lemma \ref{dim}, $\FPdim(X_g)\geq \sqrt{2}$. Hence $\FPdim(\C_g)\geq \FPdim(X_g)^2+\FPdim(X \otimes X_g)^2\geq 6$. Since $\FPdim(\C_g)=\FPdim(\I)$, we conclude that $X_g,X \otimes X_g$ are exactly the only two simple objects in $\C_g$ and $\FPdim(X_g)=\sqrt{2}$. Since $X_g\otimes X_g^*\in \I$, the only possible decomposition of $X_g\otimes X_g^*$ is $\1\oplus \delta$, which coincides with $X\otimes X^*$. This is impossible by \cite[Lemma 4.10]{natale2012solvability}. This proves that every component $\C_g$ contains at least one invertible object.

\end{proof}

\section*{Acknowledgements}
The author is partially supported by the startup foundation for introducing talent of NUIST (Grant No. 2018R039) and the Natural Science Foundation of China (Grant No. 11201231)


\end{document}